\documentclass[12pt, a4paper, leqno]{amsart}

\usepackage[utf8]{inputenc}
\usepackage{amsmath} 
\usepackage{amsfonts}
\usepackage{amssymb}
\usepackage{txfonts}   
\usepackage{amsfonts}     
\usepackage{graphicx}     
\usepackage[dvipsnames]{xcolor}
\usepackage[colorlinks, allcolors=RedViolet]{hyperref}

\newcommand{\R}{\varmathbb{R}}
\newcommand{\Z}{\varmathbb{Z}}
\newcommand{\N}{\varmathbb{N}}
\newcommand{\Rn}{{\varmathbb{R}^n}}

\newcommand{\Ha}{\mathcal{H}}
\newcommand{\tHa}{\tilde{\mathcal{H}}}
\newcommand{\M}{\mathcal{M}}

\newcommand{\Po}{\mathcal{P}}

\newcommand{\ve}{\varepsilon}

\def\diam{\qopname\relax o{diam}}
\def\loc{\qopname\relax o{loc}}
\def\spt{\qopname\relax o{spt}}

\def\diam{\qopname\relax o{diam}}
\def\inte{\qopname\relax o{int}}

\def\phi{\varphi}

\def\L{{\textsc{n}L}}

\let\oldmarginpar\marginpar
\renewcommand\marginpar[1]{\-\oldmarginpar[\raggedleft\footnotesize #1]%
{\raggedright\footnotesize #1}}

\usepackage{amsthm}

\swapnumbers
\theoremstyle{plain}
\newtheorem{theorem}[equation]{Theorem}
\newtheorem{lemma}[equation]{Lemma}
\newtheorem{proposition}[equation]{Proposition}
\newtheorem{corollary}[equation]{Corollary}

\theoremstyle{definition}
\newtheorem{definition}[equation]{Definition}

\newtheorem{example}[equation]{Example}

\theoremstyle{remark}
\newtheorem{remark}[equation]{Remark}

\numberwithin{equation}{section}

\title{On Lebesgue points and measurability with Choquet integrals}

\author{Petteri Harjulehto}
\address[Petteri Harjulehto]{Department of Mathematics and Statistics,
FI-00014 University of Helsinki, Finland}
\email{petteri.harjulehto@helsinki.fi}

\author{Ritva Hurri-Syrj\"anen}
\address[Ritva Hurri-Syrj\"anen]{Department of Mathematics and Statistics,
FI-00014 University of Helsinki, Finland}
\email{ritva.hurri-syrjanen@helsinki.fi}

\date{\today}

\begin{document}


\begin{abstract} 
We consider Choquet integrals with respect to dyadic Hausdorff content of non-negative functions which are not necessarily Lebesgue measurable.
We  study the theory of Lebesgue points.
The studies yield convergence results and  also a density result 
 between function spaces.
We provide examples which show sharpness of the main convergence  theorem.
These  examples give additional information  about the convergence
 in the norm  also,  namely the difference 
of the functions in this setting and continuous functions.
\end{abstract}

\maketitle

%%%%%%%%%%%%%%%%%%%%%%%%%%%%%%%%%%%%%%%%%%%%%%%%%%%%%%%
\section{Introduction}

The Choquet integral  was introduced by G.\ Choquet \cite{Cho53} and  applied  by D.\ R.\ Adams  in the study of
nonlinear potential theory \cite{Adams1986, Adams1998, Adams2015}. Later on,
J. Xiao has studied Choquet integrals extensively with Adams 
\cite{AdaX2020}
and with his other coauthors \cite{DX, SXY}.
Recently, there has arisen a new interest in Choquet integrals and their properties. We refer to
\cite{Tang, PonceSpector2020, MartinezSpector2021, OoiPhuc2022, ChenOoiSpector2024, ChenSpector2023, 
HH-S_JFA,HH-S_AAG, PonceSpector2023, HH-S_AWM, HH-S_La, HCYZ, HH-S_pp, BCRS_pp}.

Choquet integral theory  has been 
concentrated to the context when functions are  continuous  or quasicontinuous  or at least Lebesgue measurable. However, this 
 seems not to be necessary.
 As in the case  of 
the Riemann and Lebesgue integral theories  which were developed for Lebesgue measurable functions  first, and  later on,
these integral theories  emerged  when functions are not Lebesgue measurable  \cite{Hil17, Jef32, Zak66},  also now
it seems that it is worth to study properties of Choquet integrals for functions which are not necessarily Lebesgue measurable.
Recall that D.\ Denneberg does the ground work 
in his monograph  \cite{Den94}
where  he has stated and proved
properties of Choquet integrals with respect to capacities with minimal assumptions on  these capacities  as well as functions.

We are interested in Choquet integrals with respect to dyadic Hausdorff content, that is dyadic Hausdorff capacity.
We study Lebesgue point properties using Choquet integrals with respect to dyadic Hausdorff content. We pay special attention to the assumption of Lebesgue measurability.
As a by-product of results we obtain a characterization for Lebesgue measurable functions:

\begin{corollary}\label{cor:intro}
Let $n\geq 1$.
Suppose that $f:\Rn \to [0, \infty]$ is a function  such that $\int_{B} f(y) \, d\tHa^n_\infty (y) < \infty$ for all open balls $B$ in  $\Rn$. Then,  the function $f$ is Lebesgue measurable if and only if
\[
f(x) = \limsup_{r \to 0^+} \frac{1}{\tHa^n_\infty(B(x, r))}\int_{B(x,r)} f(y) \, d\tHa^n_\infty (y)
\]
for $\tHa^n_\infty$-almost every $x \in \Rn$.
\end{corollary}

In Section \ref{sec:Hausdorff} and Section \ref{sec:Choquet} we recall definitions and basic properties for dyadic Hausdorff content and
Choquet integral with respect to dyadic Hausdorf content, respectively. 
 We start Section \ref{sec:WeakType} by proving convergence results
for increasing sequences of non-negative functions when integrals are taken in Choquet sense with respect to dyadic Hausdorff content.
Later on, 
 we recall the definitions needed for study Hausdorff content maximal functions
  and the function space $\L^1(\Omega, \tHa^{\delta}_\infty)$, $0<\delta \le n$.
In Section \ref{sec:Lebesgue points} we  concentrate on properties related to Lebesgue points in this context 
 and
introduce the concept of $\L^1$-limit of continuous functions  in
Definition \ref{nlimitdef}
for non-negative functions in the space 
$\L^1(\Omega, \tHa^{\delta}_\infty)$, $0<\delta \le n$.
This will lead us study functions defined in this space and denseness of continuous functions  there.
In Section~\ref{sec:sharpness} we provide two examples of Lebesgue measurable functions from 
$\L^1(\Omega, \tHa^{\delta}_\infty)$, $0<\delta <n$, which are not 
$\L^1$-limits of continuos functions.
In particular, these examples show that 
continuous functions are not dense in  the space $\L^1(\Rn, \tHa^\delta_\infty)$
if $0<\delta< n$,
Remark \ref{rmk:dense}.
The examples also show necessity and sharpness of our assumption
 in one of our main theorems,  Theorem~\ref{thm:measurable}.

%%%%%%%%%%%%%%%%%%%%%%%%%%%%%%%%%%%%%%%%%%%%%%%%%%%%%%%%%%%%%%%%%%%%%%%%%%%%%%%%%%%
\section{Hausdorff content}\label{sec:Hausdorff}

We recall the definition of the  $\delta$-dimensional 
dyadic Hausdorff content for any given set in $\Rn$,
\cite{YangYuan08}.
Cubes in $\Rn$ are denoted by $Q$ and 
the side length of a  given cube $Q$ is written as 
$\ell (Q)$. 

We let $E$ be a set in $\Rn$, $n \ge 1$,  and suppose that
$\delta \in (0, n]$.
The  $\delta$-dimensional 
dyadic
Hausdorff content for any given set $E$ 
in $\Rn$ is defined by
\begin{equation*}
\tilde{\Ha}_\infty^{\delta} (E) := \inf \bigg\{ \sum_{i=1}^\infty \ell (Q_i)^{\delta}: E \subset \inte\big(\bigcup_{i=1}^\infty Q_i \big)\bigg\}\label{HausdorffCD}\,,
\end{equation*}
where the infimum is taken over all 
countably (or finite) collections of dyadic cubes such that the interior of the union of the cubes covers $E$. 
If $k\in\Z$ is fixed, we
write
 $\mathcal{Q}_k$ for a collection of cubes whose vertices are from the lattice $(2^{k} \Z)^n$. This means that
  the side length  of cubes is $2^k$  and each cube is  congruent to  a cube $[0, 2^k)^n$.
  It is customary to call these cubes half open.  
By dyadic cubes we mean an union of collection of cubes  $\bigcup_{k \in \Z} \mathcal{Q}_k$.

The cube covering is called
a dyadic cube covering or dyadic cube cover.

 The  $\delta$-dimensional dyadic Hausdorff content 
$\tilde{\Ha}_\infty^{\delta}$ satisfies  the following  useful properties, which means that it is a Choquet capacity.
\begin{enumerate}
\item[(H1)]  $\tilde{\Ha}_\infty^{\delta} (\emptyset) =0$;
\item[(H2)](\emph{monotonicity}) if $A \subset B$, then $\tilde{\Ha}_\infty^{\delta}  (A) \le     \tilde{\Ha}_\infty^{\delta}        (B)$;
\item[(H3)](\emph{subadditivity}) if $(A_i)$ is any sequence of sets, then
\[
\tilde{\Ha}_\infty^{\delta}
\Big(\bigcup_{i=1}^\infty A_i \Big)
\le  \sum_{i=1}^\infty         \tilde{\Ha}_\infty^{\delta}        (A_i);
\]
\item[(H4)] if $(K_i)$ is a decreasing sequence of compact sets, then 
\[
  \tilde{\Ha}_\infty^{\delta}              \Big(\bigcap_{i=1}^\infty K_i \Big)
= \lim_{i \to \infty}       \tilde{\Ha}_\infty^{\delta}           (K_i);
\]
\item[(H5)] if $(E_i)$ is an increasing sequence of  sets, then 
\[
  \tilde{\Ha}_\infty^{\delta}     \big(\bigcup_{i=1}^\infty E_i \big)
= \lim_{i \to \infty}       \tilde{\Ha}_\infty^{\delta}        (E_i).
\]
\end{enumerate}
Moreover,  $\tilde{\Ha}_\infty^{\delta}$ is \emph{strongly subadditive}, that is
\begin{equation}\label{st1}
     \tilde{\Ha}_\infty^{\delta}              (A_1 \cup A_2) + 
     \tilde{\Ha}_\infty^{\delta}
     (A_1 \cap A_2) \le 
    \tilde{\Ha}_\infty^{\delta} (A_1) +
 \tilde{\Ha}_\infty^{\delta}(A_2)
\end{equation}
for all  subsets $A_1, A_2$ in $\Rn$. 

For  the validity of these properties we refer to 
\cite[Theorem 2.1, Propositions 2.3 and 2.4]{YangYuan08}.
We point out  that
property  \eqref{st1}  is called submodularity  by D. Denneberg   \cite[p. 16]{Den94}.

%%%%%%%%%%%%%%%%%%%%%%%%%%%%%%%%%%%%%%%%%%%%%%%%%%%%%%%

\section{Choquet integral}
\label{sec:Choquet}

 We recall the definition of Choquet integral, now with respect to  the dyadic Hausdorff content.
 Let $\Omega$ be a subset of $\Rn$, $n \ge 1$.
For any function $f:\Omega\to [0,\infty]$ the integral in the sense of Choquet with respect to  the  $\delta$-dimensional  dyadic Hausdorff content is defined by
\begin{equation}\label{IntegralDef}
\int_\Omega f(x) \, d  \tHa^\delta_\infty         := \int_0^\infty \tHa^\delta_\infty\big(\{x \in \Omega : f(x)>t\}\big) \, dt. 
\end{equation}
Since  $\tilde{\Ha}_\infty^{\delta}$ is  a monotone  set function,
the corresponding distribution function
$t \mapsto \tHa^\delta_\infty\big(\{x \in \Omega : f(x)>t\}\big)$ 
for any function $f:\Omega\to [0,\infty]$
is decreasing with respect to $t$.
By decreasing property  the  distribution function 
$t \mapsto \tHa^\delta_\infty\big(\{x \in \Omega : f(x)>t\}\big)$ is measurable
 with respect to  the 1-dimensional  Lebesgue measure. Thus, $\int_0^\infty 
 \tHa^\delta_\infty\big(\{x \in \Omega : f(x)>t\}\big) \, dt$ 
 is well defined as a Lebesgue integral. 
The right-hand side of \eqref{IntegralDef} can be understood  also as an improper Riemann integral, since a decreasing function is Riemann integrable.
 The wording {\it{any function}} 
  in the present paper 
 means that the only requirement is that the function is well defined. We emphasise that the Choquet  
 integral is well defined for any non-measurable set in $\Rn$ and
for  any non-measurable function.
A classical example of a non-measurable function is a 
characteristic function of a non-measurable set,
originally constructed  by G. Vitali \cite{Vitali}.

We use the following properties of 
the Choquet integral with respect to  the dyadic Hausdorff content frequently. For these basic properties and others we refer to
\cite{Adams1998}, \cite[Chapter 4]{Adams2015}, \cite{CerMS11}, and \cite{HH-S_pp}. 

\begin{lemma}\label{lem:Integral-basic-properties}
Suppose that $\Omega$ is a subset of $\Rn$, functions 
$f, g : \Omega \to [0, \infty]$ are non-negative, 
and 
$\delta \in (0,n]$. Then:
\begin{enumerate}
\item[(I1)] $ \displaystyle \int_\Omega a f(x) \, d \tHa_\infty^{\delta} = a \int_\Omega  f(x) \, d \tHa_\infty^{\delta}$ with any $a\ge 0$;
\item[(I2)] $\displaystyle \int_\Omega f(x) \, d \tHa_\infty^{\delta}=0$ if and only if $f(x)=0$  for $\tHa_\infty^{\delta}$-almost every $x\in \Omega$;
\item[(I3)] if $E\subset \Rn$, then $\displaystyle \int_\Omega \chi_E(x) \, d \tHa_\infty^{\delta} = \tHa_\infty^{\delta}(\Omega \cap E)$;
\item[(I4)] if $A\subset B \subset \Omega$, then $\int_A f(x) \, d \tHa_\infty^{\delta} \le \int_B f(x) \, d \tHa_\infty^{\delta}$;
\item[(I5)] if $0\le f\le g$, then $\displaystyle \int_\Omega f(x) \, d \tHa_\infty^{\delta}\le \int_\Omega g(x) \, d \tHa_\infty^{\delta}$.
%\item[(I6)] $\displaystyle \int_\Omega f(x)+g(x) \, d \tHa_\infty^{\delta} \le 2\Big(\int_\Omega f(x) \, d \tHa_\infty^{\delta} + \int_\Omega g(x) \, d \tHa_\infty^{\delta}\Big)$;
%\item[(I7)] $\displaystyle \int_\Omega f(x)g(x) \, d \tHa_\infty^{\delta} \le 2\Big(\int_\Omega f(x)^p \, d \tHa_\infty^{\delta}\Big)^{1/p} \Big( \int_\Omega g(x)^q \, d \tHa_\infty^{\delta}\Big)^{1/q}$   when   $p, q>1$ and  $\frac{1}{p}+\frac{1}{q}=1$.
\end{enumerate}
\end{lemma}

\begin{remark}\label{lemma_a}
Note that if  $\Omega \subset \Rn$ is a Lebesgue measurable set, then
 by  \cite[Remark 3.4]{HH-S_pp} 
 there 
exists a constant  $c(n)>0$ such that 
\begin{equation}\label{compare_int}
\frac{1}{c(n)}  \int_\Omega |f(x)|  \, d \tHa^{ n}_\infty \le \int_\Omega |f(x)| \, dx \le c(n)  \int_\Omega |f(x)|  \, d \tHa^{ n}_\infty
\end{equation}
for all Lebesgue measurable  functions $f:\Omega \to [-\infty, \infty]$.
\end{remark}

Recall that a Choquet integral with respect to a set function $H: \Po(\Rn) \to [0, \infty]$           
is sublinear, if the inequality
\begin{equation*}
\int_\Omega \sum_{i=1}^{\infty} f_i(x) d H\le  K \sum_{i=1}^{\infty} \int_\Omega  f_i(x) d H
\end{equation*}
holds with $K=1$ for 
all sequences $(f_i)$ of functions $f_i:\Omega \to [0, \infty]$. 
By \cite[Chapter 6]{Den94},
a necessary condition for  the Choquet integral to be sublinear is  that the corresponding set function $H$ is strongly subadditive,
that is 
\begin{equation*}%\label{st}
H(A_1 \cup A_2) + H(A_1 \cap A_2) \le H(A_1) + H(A_2)
\end{equation*}
for all  subsets $A_1$ and  $A_2$ in $\Rn$. 
Since the  dyadic Hausdorff content  $\tilde{\Ha}_\infty^{\delta}$ is  a monotone and 
strongly subadditive set function by (H2) and \eqref{st1} 
and we use only non-negative functions,
Denneberg's result
\cite[Theorem 6.3, p.~75]{Den94} shows 
that Choquet integral with respect to the  dyadic Hausdorff content  $\tilde{\Ha}_\infty^{\delta}$ is sublinear.

\begin{theorem}\label{thm:sublin} \cite[Theorem 6.3]{Den94}.
If $\Omega$ is a subset of $\Rn$ and  $\delta \in (0, n]$, then
for 
all sequences $(f_i)$ of non-negative functions $f_i:\Omega \to [0, \infty]$ 
\begin{equation*}
\int_\Omega \sum_{i=1}^{\infty} f_i(x) d \tilde{\Ha}_\infty^{\delta} \le  \sum_{i=1}^{\infty} \int_\Omega  f_i(x) d \tilde{\Ha}_\infty^{\delta}.
\end{equation*}
\end{theorem}

 We point out that Denneberg  stated and proved his theorem
\cite[Theorem 6.3]{Den94} for a more general setting.
  
%%%%%%%%%%%%%%%%%%%%%%%%%%%%%%%%%%%%%%%%%%%%%%%%%%%%%%%%%

\section{Auxiliary convergence results}\label{sec:WeakType}

We prove a monotone convergence result and Fatou's lemma for non-negative functions in this setting. For a similar convergence theorem we refer to
\cite[Theorem 8.1]{Den94} where the situation is more general. On the other hand, convergence results with different additional assumptions than ours have been proved in
\cite{Kawabe2019} and \cite{PonceSpector2023}, for example.

\begin{proposition}\label{lema:convergence}
Suppose that $\delta \in(0, n]$.
 If $(f_i)$ is a sequence of non-negative  increasing functions $f_i:\Omega\to [0,\infty]$ and 
$f(x) = \lim_{i \to \infty} f_i (x)$ for every $x \in \Omega$, then
\[
\lim_{i \to \infty} \int_\Omega f_i \, d \tHa_\infty^{\delta} =
\int_\Omega f \, d \tHa_\infty^{\delta}.
\]
\end{proposition}

\begin{proof}
Since the sequence $(f_i)$ is increasing,  the monotonicity property (H2) of the set function $\tHa_\infty^{\delta}$  yields that
\begin{equation}\notag
\tHa_\infty^{\delta} \big(\{x \in \Omega : f_i(x) >t\} \big) \le \tHa_\infty^{\delta} \big(\{x \in \Omega : f_{i+1}(x) >t\} \big)
\mbox{ for all } t\geq 0\,.
\end{equation}
This means that the set function
$t \mapsto \tHa_\infty^{\delta} \big(\{x \in \Omega : f_i(x) >t\} \big) $ is increasing  with respect to the index  $i$. 
Hence, the
Lebesgue monotone convergence results in $\R$, 
 for example \cite[2.14]{Folland},
imply that
\[
\begin{split}
\lim_{i \to \infty} \int_\Omega f_i \, d \tHa_\infty^{\delta}
&= \lim_{i \to \infty}\int_0^\infty \tHa_\infty^{\delta} \big(\{x \in \Omega : f_i(x) >t\} \big) \, dt\\
&= \int_0^\infty \lim_{i \to \infty} \tHa_\infty^{\delta} \big(\{x \in \Omega : f_i(x) >t\} \big) \, dt\,.
\end{split}
\]
By property (H5) 
we obtain the claim. Namely,
\begin{align*}
&\int_0^\infty \lim_{i \to \infty} \tHa_\infty^{\delta} \big(\{x \in \Omega : f_i(x) >t\} \big) \, dt
= \int_0^\infty  \tHa_\infty^{\delta} \Big( \bigcup_{i=1}^\infty \{x \in \Omega : f_i(x) >t\} \Big) \, dt \\
&\quad=\int_0^\infty \tHa_\infty^{\delta} \big(\{x \in \Omega : f(x) >t\} \big) \, dt
= \int_\Omega f \, d \tHa_\infty^{\delta}. \qedhere 
\end{align*}
\end{proof}

The previous convergence result 
gives  Fatou's lemma.

\begin{proposition}\label{cor:Fatou}
Let $\delta \in(0,n]$. If
$(f_i)$ is a sequence of non-negative  functions defined on $\Omega$, then
\[
 \int_\Omega \liminf_{i \to \infty} f_i \, d \tHa_\infty^{\delta} \le \liminf_{i \to \infty} \int_\Omega f_i \, d \tHa_\infty^{\delta}\,.
\]
\end{proposition}

\begin{proof}
Let us write that $g_k(x) := \inf_{j \ge k} f_j(x)$. Then $(g_k)$ is an increasing
sequence and $\lim_{k \to \infty}g_k(x) = \liminf_{i \to \infty} f_i(x)$. 
By Proposition~\ref{lema:convergence} we obtain
\[
\begin{split}
 \int_\Omega \liminf_{i \to \infty} f_i \, d \tHa_\infty^{\delta}
 &=  \int_\Omega \lim_{k \to \infty} g_k \, d \tHa_\infty^{\delta}
 = \lim_{k \to \infty} \int_\Omega  g_k \, d \tHa_\infty^{\delta}
 \le  \liminf_{k \to \infty} \int_\Omega  f_k \, d \tHa_\infty^{\delta}.
\end{split}
\]
\end{proof}

We recall a definition of  a space of functions which
are not necessarily  Lebesgue measurable,  \cite[Remark 3.11]{HH-S_pp} .
 \begin{definition}\label{nL1space}
Let   $\Omega$ be a subset of $\Rn$, $n\geq 1$, and $0<\delta\le n$. We write
\[
\L^1(\Omega, \tHa^{\delta}_\infty):= \Big\{f: \Omega \to [-\infty, \infty] :   \int_\Omega |f|  d \Ha^{\delta}_\infty < \infty \Big\}.
\]
\end{definition}
Here, by a function we mean an equivalent class of functions that coincide $\tHa^\delta_\infty$-almost everywhere.
Property (I1) and  Theorem \ref{thm:sublin} 
imply that  $\L^1(\Omega, \Ha^{\delta}_\infty)$
is an $\R$-vector space. 
We denote
\begin{equation}\label{one_norm}
\|f\|_1 :=  \int_\Rn |f| \, d\tHa^\delta_\infty.
\end{equation}
 Theorem~\ref{thm:sublin} together (I1) yields the following proposition.  
 
 \begin{proposition}\label{norm}
 The notion
 $\|\cdot\|_1$ is a norm in $\L^1(\Omega, \tHa^{\delta}_\infty)$. 
\end{proposition}
More information about $\L^1(\Omega, \Ha^{\delta}_\infty)$-spaces can be found in \cite[Chapter~3]{HH-S_pp}.
By $f \in \L^1_{\loc} (\Rn, \tHa^{\delta}_\infty)$ we mean that 
 $f \in \L^1(K, \tHa^{\delta}_\infty)$
whenever $K\subset \Rn$ is a bounded set.

\begin{theorem}\label{lem:point-wise_convergence}
Let $\delta \in(0,n]$ and $f \in \L^1 (\Rn, \tHa^{\delta}_\infty)$.
If $(f_i)$ is a sequence of $\L^1 (\Rn, \tHa^{\delta}_\infty)$-functions such that $\lim_{i \to \infty}\|f-f_i\|_1 =0$, then there exists a subsequence 
$(f_{i_j})$ which converges to $f$ pointwise
for 
$\tHa^\delta_\infty$-almost every $x \in \Rn$.
\end{theorem}

\begin{proof}
Since $\|f_i - f_j\|_1 \le \|f_i -f\|_1 + \|f_j -f\|_1$
 by Proposition \ref{norm},
the standard argument yields 
that $(f_i)$ is a Cauchy sequence with respect to \ $\|\cdot\|_1$. Let us choose a suitable subsequence of $(f_i)$. 
Choose $i_1$ so that $\|f_i - f_j\|<\tfrac12$ when $i, j \ge i_1$. Assume that $i_1, \ldots\,,i_m$ have be chosen so that
 $\|f_i - f_j\|<\tfrac1{2^m}$ when $i, j \ge i_m$. Then we choose $i_{m+1}$ so that
 $\|f_i - f_j\|<\tfrac1{2^{m+1}}$ when $i, j \ge i_{m+1}$ and $i_{m+1} > i_m$. For the subsequence $(f_{i_j})$ we have
 \[
 \| f_{i_{j+1}}- f_{i_j} \|_1 < \frac{1}{2^{i_j}}.
 \]
 Let us write
 \[
 g_m(x) := \sum_{j=1}^m |f_{i_j}(x) - f_{i_{j+1}}(x) | \quad \text{and} \quad  g(x) := \sum_{j=1}^\infty |f_{i_j}(x) - f_{i_{j+1}}(x) |.
 \]
 Since $(g_m)$ is a pointwise increasing sequence, 
  Proposition~\ref{lema:convergence}  and Theorem~\ref{thm:sublin}  imply that
 \[
 \begin{split}
 \|g\|_1 &=  \|\lim_{m \to \infty} g_m\|_1 = \lim_{m \to \infty} \|g_m\|_1 = \lim_{m \to \infty} \int_\Rn  \sum_{j=1}^m |f_{i_j}(x) - f_{i_{j+1}}(x) | \, d \tHa^\delta_\infty(x)\\
 &\le \lim_{m \to \infty}   \sum_{j=1}^m \int_\Rn |f_{i_j}(x) - f_{i_{j+1}}(x) | \, d \tHa^\delta_\infty(x)
 \le  \sum_{j=1}^\infty \frac{1}{2^{i_j}} \le 1.
 \end{split}
 \]
 This means
  that $g$ is finite $\tHa^\delta_\infty$-almost everywhere. Let us  study the telescopic series
 $f_{i_1}(x) + \sum_{k=1}^\infty (f_{i_k}(x) - f_{i_{k+1}}(x))$. 
  In the points where $g(x)<\infty$, the telescopic series converges absolutely, and 
 thus the telescopic series  also  converges $\tHa^\delta_\infty$-almost everywhere.
 Let us denote 
 by $h(x)$  the limit in the points  where the telescopic series converges,
let $h(x)=0$ otherwise.
 Then we have
 \[
 \begin{split}
 h(x) &= f_{i_1}(x) + \sum_{j=1}^\infty (f_{i_j}(x) - f_{i_{j+1}}(x))
 = \lim_{m \to \infty}\Big( f_{i_1}(x) + \sum_{j=1}^{m-1} (f_{i_j}(x) - f_{i_{j+1}}(x))\Big)\\
 &= \lim_{m \to \infty}f_{i_m}(x)
 \end{split}
 \]
 on the points where the telescopic series converges. 
 
 Let us then show that $\lim_{i \to \infty}\|h-f_{i}\|_1 =0$. Let $\ve>0$.
 Since $(f_i)$ is a Cauchy-sequence there exists $i_\ve$ such that $\|f_i-f_j\|_1 < \ve$ when $i, j \ge i_\ve$.
 For a fixed $i$ we have $f_{i_j}(x) - f_i(x) \to h(x) - f_i(x)$ as $j \to \infty$ for $\tHa^\delta_\infty$-almost every $x \in \Rn$.
 Proposition~\ref{cor:Fatou} implies that
 \[
 \begin{split}
 \|h-f_i\|_1 = \|\lim_{j \to \infty} |f_{i_j} - f_i|\|_1 \le \liminf_{j \to \infty} \|f_{i_j} - f_i\| \le \ve
 \end{split}
 \]
when $i \ge i_\ve$. 

Finally, 
the assumption $\lim_{i \to \infty}\|f-f_i\|_1 =0$ and the previous step  yield that
\begin{equation*}
\|f-h\|_1 \le \|f-f_{i}\|_1 + \|f_{i}-h\|_1 \to 0 \mbox{ as  }i \to \infty\,.
\end{equation*}
Hence, $f(x)=h(x)$ for $\tHa^\delta_\infty$-almost every $x \in \Rn$.
\end{proof}

Let $\delta\in (0,n]$. We recall that  the   Hausdorff content  centred  maximal function is defined as
\begin{equation}\label{new}
\M^\delta  f(x) := \sup_{r>0} \frac{1}{\tHa_\infty^\delta (B(x, r))} \int_{B(x, r)} |f| \, d \tHa_\infty^\delta\,,
\end{equation}
\cite{ChenOoiSpector2024}, \cite[(4.1)]{HH-S_pp}.

We need the following result, which 
comes from  
\cite[Theorem A]{BCRS_pp} by choosing $\varphi(t) = t^{\delta}$ and noting that the different Hausdorff contents are comparable \cite[Proposition 2.3]{YangYuan08}.
This maximal function for non-measurable functions in the case $\delta =n$ have been studied also  in \cite{HH-S_pp}, 
where strong-type estimates have been proved.

\begin{theorem}
\cite[Theorem A]{BCRS_pp}
\label{thm:weak-type}
Let $\delta \in(0,n]$. Then there exists a constant $c$  depending only on $\delta$ and $n$ such that  
\begin{equation}\label{BCRS_thmA}
\tHa^\delta_\infty\Big(\{x \in \Rn: \M^{\delta} f(x) > t\} \Big)
\le \frac{c}{t} \int_\Rn |f| \, d \tHa^\delta_\infty
\end{equation}
for all  $f \in \L^1_{\loc} (\Rn, \tHa^{\delta}_\infty)$ and all $t >0$.
\end{theorem}

\begin{remark}
The definition for the maximal operator in the present paper  \eqref{new}  goes back to
\cite{ChenOoiSpector2024}. For the resent results on the classical maximal  operators 
when  integrals  are taken in sense of Choquet with respect to Hausdorff content
we refer to
H. Saito, H. Tanaka,  and T. Watanabe 
\cite{SaitoTanakaWatanabe2016, Saito2019, SaitoTanakaWatanabe2019, SaitoTanaka2022}
 and H. Watanabe \cite{WatanabeH}.
\end{remark}

The following result to  the maximal operator $\M^n$   defined in \eqref{new}
when the Choquet integral is taken 
with respect to  $\tHa^\delta_\infty$, $\delta <n$  might be of independent interest.

\begin{theorem}\label{thm:weak-type-2}
Let $\delta \in(0,n)$. Then there exists a constant $c$  depending only on $\delta$ and $n$ such that  
\[
\tHa^\delta_\infty\Big(\{x \in \Rn: \M^{n} f(x) > t\} \Big)
\le \frac{c}{t} \int_\Rn |f| \, d \tHa^\delta_\infty
\]
for all  $f \in \L^1_{\loc} (\Rn, \tHa^{\delta}_\infty)$ and all $t >0$.
\end{theorem}

\begin{proof}
Property  (I1), the strong-type estimate from \cite[Theorem 4.6]{HH-S_pp},  and \cite[Proposition 2.3]{YangYuan08} yield  that
\[
\tHa^\delta_\infty\Big(\{x \in \Rn: \M^{n} f(x) > t\} \Big)
\le  \int_\Rn \frac{1}{t} \M^{n} f(x) \, d \tHa^\delta_\infty
\le \frac{c}{t} \int_\Rn |f(x)] \, d \tHa^\delta_\infty,
\]
where  $c$ is a constant which depends only on  $n$ and $\delta$.
\end{proof}

We point out 
that Theorem~\ref{thm:weak-type} and Theorem \ref{thm:weak-type-2} 
give different benefits.
At first glance, Theorem~\ref{thm:weak-type} seems to be more natural.
On the other hand,
an estimate like  the inequality
\[
\int_\Omega \M^n f \, d \tHa^\delta_\infty \le \int_\Omega |f| \, d \tHa^\delta_\infty \quad \text{where} \quad \delta <n,  
\]
i.e. a strong-type version of Theorem~\ref{thm:weak-type-2}
 is beneficial to prove a $\delta$-dimensional Poincar\'e inequality; we refer to \cite{HH-S_JFA}.

We record also
a pointwise estimate for  maximal operators $\M^n$ and $\M^\delta$, where $\delta\in (0,n)$.  By \cite[(2.3) and Proposition 2.5]{HH-S_pp} 
it is known
that $\tHa^{\delta}_\infty (B(x, r)) \approx r^{\delta}$, 
%This together with \cite[Proposition 2.3]{HH-S_La}
%  implies that 
Now \cite[Proposition 2.3]{HH-S_La} together with this fact imply the inequality
\[
\frac1{\tHa^n_\infty(B(x, r))} \int_{B(x, r)} |f| \, d \tHa^n_\infty
\le c \bigg( \frac1{\tHa^\delta_\infty(B(x, r))} \int_{B(x, r)} |f|^{\frac{\delta}{n}} \, d \tHa^\delta_\infty\bigg)^{\frac{n}{\delta}}\,,
\]
 where $c$ is a constant independent of the function $f$.
Hence for all $x\in \Rn$
 \begin{equation}\label{pointwise}
 \M^n f(x) \le c \big(\M^\delta (f^{\frac{\delta}{n}})(x)\big)^{\frac{n}{\delta}}\,,
 \end{equation}
where $c$ is a constant which depends only on the dimension $n$ and $\delta$.

%%%%%%%%%%%%%%%%%%%%%%%%%%%%%%%%%%%%%%%%%%%%%%%%%%%%%%%

\section{Lebesgue points and covergence results}\label{sec:Lebesgue points}

For a non-negative function $f \in \L^1_{\loc} (\Rn, \tHa^{\delta}_\infty)$  we write shortly
\begin{equation}\label{average}
f_{B(x,r)}^\delta:= \frac{1}{\tHa_\infty^\delta (B(x,r))} \int_{B(x, r)} f \, d \tHa_\infty^\delta.
 \end{equation}

\begin{proposition}\label{prop:limsup_f_is_measurable}
Let $\delta\in(0, n]$ and  $f \in \L^1_{\loc} (\Rn, \tHa^{\delta}_\infty)$ be a  non-negative function. Then
\begin{enumerate}
\item  $x \mapsto f_{B(x,r)}^\delta$ is lower semicontinuous;
\item  $\limsup_{r \to 0^+} f_{B(x,r)}^\delta$ is Lebesgue measurable.
\end{enumerate}
\end{proposition}

\begin{proof}
We first show that the set $\{x \in \Rn : f_{B(x, r)}^\delta>t\}$ is open for all $t \in \R$. We may assume that $t \ge 0$.
Let us fix $t\ge 0$ and then take  any point 
 $x\in \{u \in \Rn : f_{B(u, r)}^\delta>t\}$.
If $(r_i)$ be a sequence of positive real numbers converging to $r$ from below, then 
Proposition~\ref{lema:convergence}  implies that
\[
\lim_{i \to \infty} \int_{B(x, r)} f \chi_{B(x, r_i)} \, d \tHa_\infty^{\delta} =
\int_{B(x, r)} f \, d \tHa_\infty^{\delta}.
\]
Thus, there exists a number $r_i$ such that
\[
\frac{1}{\tHa_\infty^\delta (B(x,r))} \int_{B(x, r_i)} f \, d \tHa_\infty^\delta >t.
\]
Let 
$\eta >0$  
be so small that if $y \in B(x, \eta)$, then $B(x, r_i) \subset B(y, r)$.
Since $\tHa_\infty^\delta (B(x,r)) = \tHa_\infty^\delta (B(y,r))$,
we obtain for every $y \in B(x, \eta)$ that
\[
\begin{split}
f_{B(y, r)}^\delta &= \frac{1}{\tHa_\infty^\delta (B(y,r))} \int_{B(y, r)} f \, d \tHa_\infty^\delta 
\ge \frac{1}{\tHa_\infty^\delta (B(y,r))} \int_{B(x, r_i)} f \, d \tHa_\infty^\delta \\
&= \frac{\tHa_\infty^\delta (B(x,r))}{\tHa_\infty^\delta (B(y,r))} \frac{1}{\tHa_\infty^\delta (B(x,r))} \int_{B(x, r_i)} f \, d \tHa_\infty^\delta >t\,.
\end{split}
\]
Hence, 
$B(x, \eta) \subset \{u\in \Rn : f_{B(u, r)}^\delta>t\}$.
The claim (2) is clear,
since the limit
function $\limsup_{r \to 0^+} f_{B(x,r)}^\delta$ is Lebesgue measurable as a pointwise limit of Lebesgue measurable functions.
\end{proof}

\begin{proposition}\label{prop:limit_is_measurable}
Let $\delta, \delta'\in(0, n]$  be given.
If $f \in \L^1_{\loc} (\Rn, \tHa^{\delta}_\infty)$, and $f(x) = \limsup_{r \to 0^+}f_{B(x,r)}^\delta$ for $\tHa^{\delta'}_\infty$-almost every $x \in \Rn$, then $f$ is Lebesgue measurable.
\end{proposition}

\begin{proof}
Suppose that 
$f(x) = \limsup_{r \to 0^+}f_{B(x,r)}^\delta$
for all $x \in \Rn\setminus E$ and $\tHa^{\delta'}_\infty(E) =0$. Then the set $E$ has   a zero 
$n$-dimensional Hausdorff  measure, and hence  it has also  a  zero $n$-dimensional Lebesgue outer measure. 
Thus the set $E$ is Lebesgue measurable and  its Lebesgue measure $|E|=0$. The function 
$f(x) = \limsup_{r \to 0^+}f_{B(x,r)}^\delta$
is Lebesgue measurable by Proposition~\ref{prop:limsup_f_is_measurable}, and thus $f$ coincides to a 
Lebesgue measurable function Lebesgue almost everywhere. Hence $f$ is 
Lebesgue measurable.
\end{proof}

From now on let us  define for each function $f \in \L^1_{\loc} (\Rn, \tHa^{\delta}_\infty)$  
a  corresponding function  $f^\delta_*$,    
\begin{equation}
\begin{split}\label{tahti}
f^\delta_*(x) &:= \limsup_{r \to 0^+} |f-f(x)|_{B(x,r)}^\delta\\ 
&= \limsup_{r \to 0^+} \frac{1}{\tHa_\infty^\delta (B(x,r))} \int_{B(x, r)} |f(y) - f(x)| \, d \tHa_\infty^\delta(y).
\end{split}
\end{equation}

The proofs for Lemmata \ref{lem:step1} -- \ref{lem:step2} are modifications of  the  classical  
case  for Lebesgue measurable functions.
The classical Lebesgue case
can be found in \cite[Section 3.4]{Folland}, \cite[Section 2.3]{Kin}, and \cite[Theorem 1.3.8]{Ziemer89}, for example.
We point out that functions in
Lemmata \ref{lem:step1} -- \ref{lem:step2} are allowed to be Lebesgue non-measurable.

\begin{lemma}\label{lem:step1}
Let $\delta\in(0, n]$, and $f, g \in \L^1_{\loc} (\Rn, \tHa^{\delta}_\infty)$. Then
\begin{enumerate}
\item $(f+g)^\delta_*(x) \le f^\delta_*(x) + g^\delta_*(x)$  for all $x \in \Rn$;
\item $f^\delta_*(x) \le \M^{\delta} f(x) + f(x)$ for all $x \in \Rn$.
\end{enumerate}
\end{lemma}

\begin{proof}
Property (I5) and Theorem~\ref{thm:sublin} give
\[
\begin{split}
(f+g)^\delta_*(x) &\le \limsup_{r \to 0^+}\frac{1}{\tHa_\infty^\delta (B(x,r))} \int_{B(x, r)} |f(y)- f(x)| + |g(y)- g(x)| \, d \tHa_\infty^\delta(y)\\
&
\le \limsup_{r \to 0^+}\frac{1}{\tHa_\infty^\delta (B(x,r))} \Big(\int_{B(x, r)} |f(y)- f(x)| \, d \tHa_\infty^\delta(y) \\
&\qquad+
\int_{B(x, r)} |g(y)-g(x)| \, d \tHa_\infty^\delta(y) \Big)\\
& \le f^\delta_*(x) + g^\delta_*(x).
\end{split}
\]
 In a similar manner,
\[
f^\delta_*(x) \le  \limsup_{r \to 0^+} \frac{1}{\tHa_\infty^\delta (B(x,r))} \int_{B(x, r)} |f(y)| +  |f(x)| \, d \tHa_\infty^\delta(y) \le \M^\delta f(x) + f(x).
\]
\end{proof}

We write $C(\Rn)$ for continuous functions defined on $\Rn$
and $C_0(\Rn)$ for continuous functions with compact support on $\Rn$.

\begin{lemma}\label{lem:step3}
Let $\delta\in(0, n]$, $f \in \L^1_{\loc} (\Rn, \tHa^{\delta}_\infty)$  and $\phi \in C(\Rn)$. Then
\begin{enumerate}
\item $\phi^\delta_*(x) =0$ for all $x \in \Rn$;
\item $(f-\phi)^\delta_*(x) = f^\delta_*(x)$ for all $x \in \Rn$.
\end{enumerate}
\end{lemma}

\begin{proof}
Let $\phi \in C(\Rn)$ and $x \in \Rn$. For every $\ve >0$ we found $\delta >0$ such that $|\phi(y) - \phi(x)|<\ve$ when $|y-x|<\delta$. Thus for all $r<\delta$ we find that $0\le \phi^\delta_*(x) < \ve$, and the property (1) follows.

%Let us first note that $(|f|)^* = f^*$ and $(-f)^* = f^*$. 
By  propery (I5), Theorem~\ref{thm:sublin}, and the previous case we obtain
\[
\begin{split}
f^\delta_*(x) &= \limsup_{r \to 0^+} \frac{1}{\tHa_\infty^\delta (B(x,r))} \int_{B(x, r)} |f(y) - f(x)| \, d\tHa_\infty^\delta(y)\\ 
&\le \limsup_{r \to 0^+} \frac{1}{\tHa_\infty^\delta (B(x,r))} \int_{B(x, r)} |f(y)-\phi(y) - f(x) + \phi(x)| \\
&\qquad\qquad \qquad\qquad\qquad\qquad + |\phi(y) - \phi(x)| \, d\tHa_\infty^\delta(y)\\
& \le (f-\phi)^\delta_*(x) + \phi^\delta_*(x)= (f-\phi)^\delta_*(x) \,.
\end{split}
\]
 Lemma~\ref{lem:step1}(1) implies that
\[
(f-\phi)^\delta_*(x) \le f^\delta_*(x) + (-\phi)^\delta_*(x)  = f^\delta_*(x). \qedhere
\]
\end{proof}

% From now on we assume that $\delta =n$ in the definition of $f^\delta_*$. In Section~\ref{sec:sharpness} we give an example which  shows  that it is natural to assume that $\delta =n$.

%\comment{Tässä kohtaa ennen valittiin, että $\delta =n$. Minusta sitä ei enää tarvita koska Danielin ja kummpaneiden tulos maksimaalifunktiolle on niin hyvä. Muutoksia on lauseessa %ja todistuksessa. Merkkasin muutoksen vain lauseeseen.}

\begin{lemma}\label{lem:step2}
Let $\delta \in (0, n]$.
If  $f \in \L^1 (\Rn, \tHa^{\delta}_\infty)$, then
\[
\tHa_\infty^\delta \Big( \{x \in \Rn : f^{\delta}_*(x) > \lambda \} \Big) \le \frac{c}{\lambda} \int_\Rn |f| \, d \tHa^\delta_\infty
\] 
for all $\lambda >0$, where $c$  is a constant  which depends only on $\delta$ and $n$.
\end{lemma}

\begin{proof}
Lemma~\ref{lem:step1}(2) implies that
\[
\begin{split}
 \{x \in \Rn : f^\delta_*(x) > \lambda \} &\subset  \{x \in \Rn : \M^{\delta} f(x) + f(x) > \lambda \}\\
 &\subset   \{x \in \Rn : \M^{\delta} f(x)  > \lambda/2 \} \bigcup  \{x \in \Rn : f(x) > \lambda/2 \}.
\end{split}
\] 
For the first set on the right hand side we use the weak-type estimate of the maximal operator, 
Theorem~\ref{thm:weak-type},  and obtain
\[
\begin{split}
\tHa^\delta_\infty \Big(\{x \in \Rn : \M^{\delta} f(x)  > \lambda/2 \} \Big) 
\le \frac{c}{\lambda} \int_\Rn f \, d \tHa^\delta_\infty.
\end{split}
\] 
For the second therm on the right hand side, property (I1) gives
\[
\tHa^\delta_\infty \Big(\{x \in \Rn : f(x)  > \lambda/2 \} \Big) 
\le  \int_\Rn \frac{2 f}{\lambda} \, d \tHa^\delta_\infty
= \frac{2}{\lambda} \int_\Rn f \, d \tHa^\delta_\infty.
\] 
Combining the previous estimates and using (H3) yield the claim.
\end{proof}

So far we have considered functions which have been allowed to be Lebesgue non-measurable.
Next we introduce an assumption that leads us to study their relationship to  Lebesgue measurable functions. 

%\comment{$C_0(\Rn)$-funktiot on korvattu $C(\Rn)$-funktioilla.}

\begin{definition}\label{nlimitdef}
Suppose that $\delta \in (0, n]$.
We say that   a function  $f \in \L^1 (\Rn, \tHa^{\delta}_\infty)$  is an 
\emph{$\L^1$-limit of  continuous functions} if there exists a sequence 
 of continuous functions $\phi_i$ such that 
\[
\int_{\Rn} |f - \phi_i| \, d \tHa_\infty^\delta \to 0
\]
as $i \to \infty$. 
\end{definition}

Since 
$\|\cdot\|_1$ is a norm in the space $\L^1(\Omega, \tHa^{\delta}_\infty)$
by Proposition ~\ref{norm}, 
the triangle inequality gives
$\|\phi_i\|_1 \le \|\phi_i - f\|_1 + \|f\|_1$. Hence we  may  assume that each function  $\phi_i$ in Definition~\ref{nlimitdef} belongs to  the space $\L^1 (\Rn, \tHa^{\delta}_\infty)$.

\begin{theorem}\label{thm:measurable}
Suppose that $\delta\in(0, n]$ and  
 a function
$f \in \L^1 (\Rn, \tHa^{\delta}_\infty)$. 
\begin{enumerate}
\item If $\delta =n$ and  $f$ is Lebesgue measurable, then $f$
is an $\L^1$-limit of continuous functions.
\item If $f$
is an $\L^1$-limit of continuous functions, then $f$ is Lebesgue measurable.
\end{enumerate}
\end{theorem}

\begin{remark}
In Section~\ref{sec:sharpness} we give two examples of Lebesgue measurable functions belonging
 to $\L^1 (\Rn, \tHa^{\delta}_\infty)$, $\delta <n$, that are not $\L^1$-limit of 
 continuous functions. Hence, the assumption $\delta=n$ in
Theorem \ref{thm:measurable}(1) is necessary.
\end{remark}

\begin{proof}[Proof of Theorem~\ref{thm:measurable}]
Assume first that $f$ is a Lebesgue measurable function. 
If
$\delta =n$,  we have
\[
\int_{\Rn} |f| \, dx \approx \int_\Rn |f| \, d \tHa^n_\infty 
\]
by Remark~\ref{lemma_a}. Here the implicit constant depends only on $n$. This yields that $f$ belongs to the 
(ordinary) Lebesgue space $L^1(\Rn)$. It is well known, see for example \cite[Theorem 2.19]{AdamsRA}, that $C_0(\Rn)$-functions are dense 
in $L^1(\Rn)$. Thus there exists a $C_0$-sequence $(\phi_i)$ such that 
$\int_\Rn |f-\phi_i| \, dx \to 0$ as $i \to \infty$. 
But hence also $\int_\Rn |f- \phi_i| \, d \tHa^n_\infty \to 0$, and thus $f$
is an $\L^1$-limit of $C_0$-functions.

 For the part (2) assume then that $f$
is a $\L^1$-limit of continuous functions. Let $(\phi_i)$ be a sequence of continuous functions converging to $f$ in $\|\cdot\|_1$. 
By Theorem~\ref{lem:point-wise_convergence} 
 there exists a subsequence $(\phi_{i_k})$ that convergences pointwise to  the function $f$. 
Since  each function  $\phi_{i_k}$ is continuous, and hence Lebesgue measurable, the function $f$ is a pointwise limit of 
Lebesgue measurable functions.  The exceptional set has  a zero $\tHa^\delta_\infty$-capacity, and thus it also have  a zero 
Lebesgue outer measure, and hence also zero Lebesgue measure.  These  facts  yield that $f$ is Lebesgue measurable.
\end{proof}

Let $\delta \in(0, n]$ and $f: \Rn \to [- \infty, \infty]$. 
We say that a function $f$ is $\tHa^\delta_\infty$-quasicontinuous if for every $\epsilon>0$ there exists an open set $O$ such that $\tHa^\delta_\infty(O) < \epsilon$ and and $f|_{\Rn\setminus O}$
is continuous.
Recall that $\tHa^\delta_\infty$-quasicontinuous functions are
Lebesgue measurable.

The authors express their gratitude to Daniel Spector for pointing out the following theorem.

\begin{theorem}\label{thm:quasicontinuity}
Let $\delta \in(0, n]$ and $f: \Rn \to [- \infty, \infty]$. Then $f$
is an $\L^1$-limit of  continuous functions if and only if 
$f$ is $\tHa^\delta_\infty$-quasicontinuous.
\end{theorem}

\begin{proof}
By \cite[Proposition 3.2]{PonceSpector2023} every $\Ha^\delta_\infty$-quasicontinuous function is an $\L^1$-limit of continuous functions. 

Assume then that $f$  is an $\L^1$-limit of continuous functions. Let $(\phi_i)$ be a sequence of continuous 
functions converging to $f$ in $\L^1(\Rn, \tHa^\delta_\infty)$. By Theorem~\ref{lem:point-wise_convergence} we may 
assume that $(\phi_i)$ converges to the function $f$ pointwise $\tHa^\delta_\infty$-almost everywhere. Let $N$ be the corresponding
exceptional set.  By the definition of $\tHa^\delta_\infty$, for every $\ve >0$ there exists 
an open set $E_\ve$ such that 
$N \subset E_\epsilon$, 
and $\tHa^\delta_\infty(E_\ve) < \ve$. For example, 
a set of type $\inte\big(\bigcup_{i=1}^\infty Q_i \big)$ works for suitable collection of cubes.

Since  the sequence $(\phi_i)$ is  a Cauchy sequence, we may assume, by taking a suitable subsequence, that
$\|\phi_i- \phi_{i+1}\|_1 < 4^{-i}$ for every $i$. 
%Let
We write that
\[
U_i := \{x \in \Rn : |\phi_i(x) - \phi_{i+1}(x)| > 2^{-i}\}
\]
and $V_j := \bigcup_{i=j}^\infty U_i$. Since each  function $\phi_i$ is continuous, the sets $U_i$ and $V_j$ are open.  
Since $2^i|\phi_i(x) - \phi_{i+1}(x)| >1$ in $U_i$ we find that
\[
\tHa^\delta_\infty(U_i) \le \int_\Rn 2^i|\phi_i(x) - \phi_{i+1}(x)| \, d \tHa^\delta_\infty(U_i) 
< 2^i 4^{-i} = 2^{-i},
\]
and thus by (H3)
\[
\tHa^\delta_\infty(V_j) \le \sum_{i=j}^\infty\tHa^\delta_\infty(U_i) < 2^{-j+1}.
\]
For all $x \in \Rn \setminus V_j$, and all $k>l> j$ we have
\[
|\phi_l(x) - \phi_k(x) | \le \sum_{i=l}^{k-1} |\phi_i(x) - \phi_{i+1}(x) | \le \sum_{i=l}^{k-1}2^{-i} \le 2^{1-l}.
\]
Thus the convergence is uniform in the set $\Rn \setminus V_j$. 

Let $\ve >0$. Then there exists   a set $E_\ve$ such that $N \subset E_\epsilon$, and $\tHa^\delta_\infty(E_\ve) < \ve/2$. 
We may choose  an index $j_0$ to be so large that $\tHa^\delta_\infty(V_{j_0}) < \ve/2$. 
Then the sequence $(\phi_i)$ converges uniformly and pointwise to the function $f$  in  the set $\Rn \setminus (E_\ve \cup V_{j_0})$. Thus the function $f$ restricted to the set  $\Rn \setminus (E_\ve \cup V_{j_0})$ is continuous. Moreover, the set $E_\ve \cup V_{j_0}$ is open and by (H3) the inequality
$\tHa^\delta_\infty(E_\ve \cup V_{j_0}) < \ve$ holds.
\end{proof}

R. Basak,   Y.-W. B. Chen,  P. Roychowdhury, and D. Spector study maximal functions and Lebesgue points  in their recent paper
\cite{BCRS_pp}. Let 
$\delta \in(0, n]$.
 In \cite[Theorem B]{BCRS_pp}
they show that 
\begin{equation*}
\lim_{r \to 0^+} \frac{1}{\tHa_\infty^\delta(B(x,r))} \int_{B(x, r)} |f(y) - f(x)| \, d \tHa_\infty^\delta(y) =0
\end{equation*} 
holds $\tHa^\delta_\infty$-almost everywhere 
whenever  a function  $f$ is $\tHa^\delta_\infty$-quasicontinuous. 
Instead of quasicontinuity we assume that the function is an $\L^1$-limit of 
continuous functions. By Theorem~\ref{thm:quasicontinuity} the first part of the following theorem is the same as  \cite[Theorem B]{BCRS_pp}.  

\begin{theorem}\label{thm:main}
Let $\delta \in (0, n]$.
Suppose  that  a function  $f \in \L^1 (\Rn, \tHa^{\delta}_\infty)$  is an $\L^1$-limit of 
continuous functions. Then,
\begin{equation}\label{thm:main-01}
\limsup_{r \to 0^+} \frac{1}{\tHa_\infty^{\delta}(B(x,r))} \int_{B(x, r)} |f(y) - f(x)| \, d \tHa_\infty^{\delta}(y) =0
\end{equation} 
for $\tHa^\delta_\infty$-almost every $x \in \Rn$.

If  the above function $f$ is also  non-negative, then 
\begin{equation}\label{thm:main-1}
f(x) = \limsup_{r \to 0^+} \frac{1}{\tHa_\infty^\delta(B(x,r))}\int_{B(x,r)} f(y) \, d \tHa^\delta_\infty(y)
\end{equation}
for $\tHa^\delta_\infty$-almost every $x \in \Rn$.
\end{theorem}

Note that by Theorem~\ref{thm:measurable} the function $f$ in Theorem~\ref{thm:main} is Lebesgue measurable.
Proposition~\ref{prop:limit_is_measurable} shows that the second claim in Theorem~\ref{thm:main} cannot hold for Lebesgue non-measurable functions.

\begin{remark}
If   equation \eqref{thm:main-1} is valid for  some  $x \in \Rn$, then for this $x$ 
\begin{equation*}
|f(x)| \le \M^\delta f(x)\,.
\end{equation*}
\end{remark}

Although our proof for the first part of Theorem~\ref{thm:main} is standard and 
similar to the proof of \cite[Theorem B]{BCRS_pp}, we include the proof for the  convenience of readers.

\begin{proof}[Proof of Theorem~\ref{thm:main}]
Let $(\phi_i)$ be a sequence of continuous functions  converging to $f$ in  the norm $\|\cdot\|_1$.
Then, by Lemma~\ref{lem:step3}(2) and  Lemma \ref{lem:step2}  
\[
\begin{split}
\tHa^\delta_\infty \Big(\{x \in \Rn : f^\delta_*(x)  > \lambda \} \Big) 
&=\tHa^\delta_\infty \Big(\{x \in \Rn : (f-\phi_i)^\delta_*(x)  > \lambda \} \Big)\\
&\le \frac{c}{\lambda}  \int_\Rn |f-\phi_i| \, d \tHa^\delta_\infty\,,
\end{split}
\]
 where $c$ is a constant which depends only on $n$ and $\delta$.
Letting $i \to \infty$  here
implies that $\tHa^\delta_\infty \Big(\{x \in \Rn : f^\delta_*(x)  > \lambda \} \Big) =0$
for every $\lambda >0$. Thus property (H3) yields that
\[
\begin{split}
\tHa^\delta_\infty \Big(\{x \in \Rn : f^\delta_*(x)  >0 \} \Big)
&= \tHa^\delta_\infty \Big(\bigcup_{i=1}^\infty \{x \in \Rn : f^\delta_*(x)  >1/i \} \Big)\\
&\le \sum_{i=1}^\infty \tHa^\delta_\infty \Big(\{x \in \Rn : f^\delta_*(x)  > 1/i \} \Big) =0. 
\end{split}
\]

 Let us then prove the second part of the theorem.
The assumption $f\geq 0$, property (I5), and Theorem~\ref{thm:sublin}  imply that
\[
\begin{split}
\limsup_{r \to 0} f_{B(x,r)}^\delta&= \limsup_{r \to 0} \frac{1}{\tHa_\infty^n (B(x,r))} \int_{B(x, r)} |f(y)| \, d \tHa_\infty^\delta(y)\\ 
&= \limsup_{r \to 0} \frac{1}{\tHa_\infty^\delta (B(x,r))} \int_{B(x, r)} |f(y) - f(x) + f(x)| \, d \tHa_\infty^\delta(y)\\
&\le \limsup_{r \to 0} \frac{1}{\tHa_\infty^\delta (B(x,r))} \int_{B(x, r)} |f(y) - f(x)|  \, d \tHa_\infty^\delta(y)
+  |f(x)|\\ 
&= f^\delta_*(x) + f(x).
\end{split}
\]
By the part (1)
we have $f^\delta_*(x) =0$ for $\tHa^\delta_\infty$-almost every $x \in \Rn$. Hence,
$\limsup_{r \to 0} f_{B(x,r)}^\delta \le f(x)$ for $\tHa^\delta_\infty$-almost every $x \in \Rn$. On the other hand,
\[
\begin{split}
f(x) &= \limsup_{r \to 0} \frac{1}{\tHa_\infty^\delta (B(x,r))} \int_{B(x, r)} |f(x)| \, d \tHa_\infty^\delta(y)\\
&=  \limsup_{r \to 0} \frac{1}{\tHa_\infty^\delta (B(x,r))} \int_{B(x, r)} |f(x)-f(y) + f(y)| \, d \tHa_\infty^\delta(y)\\
&\le f^\delta_*(x) + \limsup_{r \to 0} \frac{1}{\tHa_\infty^\delta (B(x,r))} \int_{B(x, r)} f(y)  \, d \tHa_\infty^\delta(y).
\end{split}
\]
Since   
$f^\delta_*(x) =0$ for $\tHa^\delta_\infty$-almost every $x \in \Rn$ by (1), we have the inequality  $f(x) \le \limsup_{r \to 0} f_{B(x,r)}^\delta$ for $\tHa^\delta_\infty$-almost every $x \in \Rn$ and the claim (2) follows also.
\end{proof}

\begin{remark}
Both assumptions,  $f \in \L^1 (\Rn, \tHa^{\delta}_\infty)$  and $f$ is an $\L^1$-limit of 
continuous functions
in Theorem~\ref{thm:main} can be relaxed since Lebesgue points have a local nature. Instead of  having
$f \in \L^1 (\Rn, \tHa^{\delta}_\infty)$ we may assume that $f \in \L^1 (B, \tHa^{\delta}_\infty)$ for every ball 
$B\subset \Rn$. And instead of  supposing that   $f$ is an $\L^1$-limit of continuous functions we may assume that in every 
ball the function $f$ is a limit of  some continuous functions. 
 Namely, the space $\Rn$ 
can be covered 
by countable many balls $(B_i)$  and in  each ball we  can
take a cut of function $\phi_i \in C^\infty_0(\Rn)$ such that $0 \le \phi_i \le 1$, $\phi_i(x) =1$ in  $B_i$, 
$\spt\phi_i \subset 2B_i$.
Then,  this function $\phi_i f$ is an $\L^1$-limit of continuous functions. Now, Theorem~\ref{thm:main} can be applied to
 the  function $f \phi_i$. 
Hence, we obtain the Lebesgue point property \eqref{thm:main-01} in   each ball $B_i$. Finally,
 subadditivity of  the $\delta$-dimensional Hausdorff content  $\tHa_\infty^\delta$  gives that  property \eqref{thm:main-01} is valid for 
$\tHa^\delta_\infty$-almost every $x \in \Rn$.
\end{remark}

 If we consider the space $\L^1 (\Rn, \tHa^{n}_\infty)$ in 
Theorems~\ref{thm:measurable} and \ref{thm:main},
that is, if we choose 
$\delta =n$ there in the space $\L^1 (\Rn, \tHa^{\delta}_\infty)$, then
Theorems~\ref{thm:measurable} and \ref{thm:main} imply Corollary \ref{cor:main}.

\begin{corollary}\label{cor:main}
Suppose that  a function  $f \in \L^1 (\Rn, \tHa^{n}_\infty)$  is Lebesgue measurable. Then,
\[
\limsup_{r \to 0^+} \frac{1}{\tHa_\infty^n(B(x,r))} \int_{B(x, r)} |f(y) - f(x)| \, d \tHa_\infty^n(y) =0
\] for $\tHa^n_\infty$-almost every $x \in \Rn$.

If in additionally $f$ is non-negative, then it also holds that
\begin{equation*}\label{cor:main-1}
f(x) = \limsup_{r \to 0^+} \frac{1}{\tHa_\infty^n(B(x,r))}\int_{B(x,r)} f(y) \, d \tHa^n_\infty(y)
\end{equation*}

for $\tHa^n_\infty$-almost every $x \in \Rn$.
\end{corollary}

Note that Proposition~\ref{prop:limit_is_measurable} shows that the second claim of Theorem~\ref{thm:main} cannot hold for Lebesgue non-measurable functions.

 Finally we prove the corollary in the introduction.

\begin{proof}[Proof of Corollary~\ref{cor:intro}]
Assume first that $f$ is Lebesgue measurable.
Let us write for every $i \in \N$ that
\[
f_i (x) := \chi_{B(0, i)} (x) f(x).
\]
Then, $f_i \in \L^1(\Rn, \tHa^n_\infty)$, and it is a non-negative Lebesgue measurable function.
Corollary~\ref{cor:main} yields that
\[ 
\begin{split}
f_i(x) &= \limsup_{r \to 0^+} \frac{1}{\tHa_\infty^n(B(x,r))}\int_{B(x,r)} f_i(y) \, d \tHa^n_\infty(y)\\
&=\limsup_{r \to 0^+} \frac{1}{\tHa_\infty^n(B(x,r))}\int_{B(x,r)} f(y) \, d \tHa^n_\infty(y)
\end{split}
\] 
for all $x \in B(0, i/2)\setminus E_i$ where $\tHa^n_\infty (E_i) =0$. 
Hence,
\[ 
f(x) = \limsup_{r \to 0^+} \frac{1}{\tHa_\infty^n(B(x,r))}\int_{B(x,r)} f(y) \, d \tHa^n_\infty(y)
\] 
for all $x \in \Rn \setminus E$, where $E = \bigcup_{i=1}^\infty E_i$.   The subadditivity of the Hausdorff content  (H3) implies that
\[
\tHa^n_\infty (E) \le \sum_{i=1}^\infty  \tHa^n_\infty (E_i) =0.
\]

The other direction follows from   Proposition~\ref{prop:limit_is_measurable}.
\end{proof}

%%%%%%%%%%%%%%%%%%%%%%%%%%%%%%%%%%%%%%%%%%%%%%%%%%%%%%%%%%%%%
\section{Counter examples}\label{sec:sharpness}

In Corollary~\ref{cor:main} we chose the Hausdorff content dimension $\delta$ to be equal to the dimension of the space, that is $\delta =n$.
In the present section we show 
that the following statement  is false:
\begin{quote}
\emph{Let $\delta \in (0, n)$.
If $f \in \L^1 (\Rn, \tHa^{\delta}_\infty)$  is  a Lebesgue measurable function, then}
$f(x) = \limsup_{r \to 0^+} f_{B(x, r)}^{\delta}$ for $\tHa^\delta_\infty$-almost every $x \in \Rn$.
\end{quote}
The problem  is that  here  the exceptional set, which has a  zero $\tHa^\delta_\infty$-content
 with $0<\delta<n$, is too small in capacity sense.

\begin{example}[Characteristic function of a $n$-dimensional ball]\label{exa:pallo}
  Let $\delta \in (0, n-1]$. We define $B:= \{x\in \Rn: x_1^2 + \ldots + x_n^2  < 1\}$, and then  take a characteristic function 
$\chi_B:\Rn \to [0, \infty]$ so that $\chi_B(x) = 1$ if $x \in B$ and $\chi_B(x) =0$  otherwise. 
Then $\chi_B$ is  a Lebesgue measurable function and $\chi_B \in \L^1  (\Rn, \tHa^{\delta}_\infty)$. 
 For every $x \in \partial B$  and $0<r<\frac12$ there exists $y \in B$ such that $B(y, \frac12 r) \subset B$.
This means that
\[
(\chi_B)_{B(x, r)}^{\delta} = \frac{1}{\tHa^{\delta}_\infty (B(x, r))} \int_{B(x, r)} \chi_B \, d \tHa^{\delta}_\infty
\ge \frac{\tHa^{\delta}_\infty (B(y, \frac12 r))}{\tHa^{\delta}_\infty (B(x, r))}.
\]
By \cite[(2.3) and Proposition 2.5]{HH-S_pp} we know that $\tHa^{\delta}_\infty (B(x, r)) \approx r^{\delta}$, 
where the implicit constant depends only on the dimension $n$.
Thus, 
\[
(\chi_B)_{B(x, r)}^{\delta} 
\ge c_1 \frac{(\frac12r)^{\delta}}{r^{\delta}} = c_1 2^{-\delta} >0\,.
\]
In conclusion, we have   that the function $\chi_B$ is Lebesgue measurable, $\chi_B \in \L^1 (\Rn, \tHa^{\delta}_\infty)$, 
 and
\[
\chi_B(x)=0 < c_1 2^{-\delta}  \le \limsup_{r \to 0^+} (\chi_B)_{B(x, r)}^{\delta}.
\]
  And  this  holds for all  points $x \in \partial B$.      
 Since $0<\delta \le n-1$, the $\delta$-dimensional Hausdorff content 
 $\tHa^\delta_\infty (\partial B) $  is positive.
 Thus, the  claim  in the beginning of this chapter  is not true.
 
 Let us look \eqref{thm:main-1} for $f=\chi_B$. The present example shows  then that equation
 \eqref{thm:main-1} does not hold,
 if $f \in \L^1 (\Rn, \tHa^{\delta}_\infty)$ and if $0 < \delta \le n-1$. Hence, by Theorem~\ref{thm:main}
 the function $\chi_B \in \L^1 (\Rn, \tHa^{\delta}_\infty)$ is not 
an $\L^1$-limit of continuous functions  if $0 < \delta \le n-1$.
\end{example}

\begin{example}[Characteristic function of a  modification of the von Koch's snowflake in plane]\label{exa:Koch}
Let us fix $n=2$ and  let $\delta\in (1, 2)$.  The von Koch curve construction can be modify 
so that the resulting  curve has  a Hausdorff dimension of $\delta$ and  a positive $\delta$-dimension Hausdorff measure, see  \cite[Section~3]{HHL06}. 
This happens by dividing a line segment $K$ into four equally long segments, and  replacing each subsegment with 
a segment of length $s|K|$, 
 for a fixed  $s \in(\frac14, \frac12)$, in such a way that the segments $K_0$ and $K_3$ are subsets 
of $K$ at opposite ends of $K$, and $K_1$ and $K_2$ are the sides of an isosceles triangle whose base is $K \setminus (K_0 \bigcup K_3)$. This gives a curve 
 with the Hausdorff dimension  $\frac{\log(4)}{\log(1/s)}$. 

%\begin{figure}[ht!]
%\includegraphics[width=8cm]{Koch1.png}
%\includegraphics[width=4cm]{Koch2.png}
%\caption{Pictures on the left-hand side and in the middle : An example of the construction of the modified von Koch's curve, where $s = \frac38$. The picture is a modification of \cite[Figure~1]{HHL06}. 
%The right-hand side domain: An 
%open set $\Omega$ bounded by three modified von Koch's curves. }
%\end{figure}

By joining three these curves together we obtain the von Koch snowflake $F$, which bounds an open set  $\Omega$.
This means that
$F = \partial \Omega$.
We show that for   all points  $x \in F$ we have
\[
\chi_\Omega(x) \neq \limsup_{r \to 0^+} (\chi_\Omega)_{B(x, r)}^{\delta}.
\]
Clearly $\chi_\Omega(x)=0$ for all $x \in F$. By the construction of  the von Koch curve there exists $c \in(0, 1)$, depending on $\delta$, such that every $x \in F$ and every $0<r<\diam(\Omega)/2$
there exists $y \in \Omega$ such that 
$B(y, cr) \subset \Omega \cap B(x, r)$. 
  Calculations as in Example~\ref{exa:pallo} imply that
\[
\chi_\Omega(x)=0 < c_1 c^{\delta}  \le \limsup_{r \to 0^+} (\chi_\Omega)_{B(x, r)}^{\delta}
\]
for  all  points $x \in F$, and $\tHa^\delta_\infty (F) >0$.
Moreover, 
 the characteristic function
$\chi_\Omega$ is  a Lebesgue measurable  function which belongs to  $\L^1 (\Rn, \tHa^{\delta}_\infty)$. 
 Thus \eqref{thm:main-1} does not hold, and hence 
by Theorem~\ref{thm:main} the function $\chi_\Omega \in \L^1 (\R^2, \tHa^{\delta}_\infty)$ 
is not an $\L^1$-limit of continuous functions.
\end{example}

\begin{remark}\label{rmk:dense}
Let $0<\delta <n$  be fixed.
Examples \ref{exa:pallo} and \ref{exa:Koch} 
show that in the function space
 $\L^1(\Rn, \tHa^\delta_\infty)$  there exist Lebesgue measurable functions  
 which  are not $\L^1$-limits of continuous functions. This implies two
 remarkable conclusions: 
 \begin{enumerate}
\item if $\delta <n$, then continuous functions are not dense in $\L^1(\Rn, \tHa^\delta_\infty)$;
\item the assumption $\delta =n$ in Theorem~\ref{thm:measurable} (1) is necessary.
\end{enumerate}
\end{remark}

%%%%%%%%%%%%%%%%%%%%%%%%%%%%%%%%%%%%%%%%%%%%%%%%%%%%%%%%%%%%%
%%%%%%%%%%%%%%%%%%%%%%%%%%%%%%%%%%%%%%%%%%%%%%%%%%%%%%%%%%%%%
%%%%%%%%%%%%%%%%%%%%%%%%%%%%%%%%%%%%%%%%%%%%%%%%%%%%%%%%%%%%
%%%%%%%%%%%%%%%%%%%%%%%%%%%%%%%%%%%%%%%%%%%%%%%%%%%%%%%%%%%%
%%%%%%%%%%%%%%%%%%%%%%%%%%%%%%%%%%%%%%%%%%%%%%%%%%%%%%%%%%%%
\bibliographystyle{amsalpha}

\end{document}